\documentclass[12pt, reqno]{amsart}

\usepackage{amsmath}
\usepackage{amsthm}
\usepackage{amssymb}
\usepackage{amsrefs}
\usepackage{mathrsfs}
\usepackage[usenames]{color}

\numberwithin{equation}{section}
\newtheorem{theorem}{Theorem}[section]
\newtheorem{lemma}[theorem]{Lemma}
\newtheorem{corollary}[theorem]{Corollary}
\newtheorem{remark}[theorem]{Remark}
\newtheorem{proposition}[theorem]{Proposition}

 \newcommand{\LL}{\ensuremath{\mathcal{L}}}
 \newcommand{\HH}{\ensuremath{\mathcal{H}}}
\newcommand{\DDD}{\ensuremath{\mathcal{D}}}
\newcommand{\DD}{\ensuremath{\mathbb{D}}}
\newcommand{\RR}{\ensuremath{\mathbb{R}}}
\newcommand{\NN}{\ensuremath{\mathbb N}}

\newcommand{\Id}{\ensuremath{\mathrm{Id}}}
\newcommand{\Bid}{\ensuremath{\mathbf{B}}}
 \newcommand{\Leb}{\ensuremath{\mathbf{L}}}
 
\newcommand{\BB}{\ensuremath{\mathcal{B}}}
\newcommand{\FF}{\ensuremath{\mathbb{F}}}
\newcommand{\GG}{\ensuremath{\mathcal{G}}}
\newcommand{\CC}{\ensuremath{\mathbb{C}}}
\newcommand{\XX}{\ensuremath{\mathbb{X}}}
\newcommand{\YY}{\ensuremath{\mathbb{Y}}}
\newcommand{\xx}{\ensuremath{\mathbf{x}}}

\newcommand{\kk}{\ensuremath{\mathbf{k}}}

\DeclareMathOperator{\supp}{supp}
\DeclareMathOperator{\spn}{span}
\DeclareMathOperator{\sgn}{sign}

\title[Asymptotic greediness
 of the Haar system in $L_{p}$]{Asymptotic greediness
 of the Haar system in the spaces $L_{p}[0,1]$, $1<p<\infty$}
\author[F. Albiac]{Fernando Albiac}
\address{Mathematics Department\\ 
Universidad P\'ublica de Navarra\\
Campus de Arrosad\'{i}a\\
Pamplona\\ 
31006 Spain}
\email{fernando.albiac@unavarra.es}

\author[J. L. Ansorena]{Jos\'e L. Ansorena}
\address{Department of Mathematics and Computer Sciences\\
Universidad de La Rioja\\ 
Logro\~no\\
26004 Spain}
\email{joseluis.ansorena@unirioja.es}

\author[P.\ M.  Bern\'a]{Pablo M. Bern\'a}
\address{Pablo M. Bern\'a\\ Mathematics Department, Universidad Aut\'onoma de Madrid\\ Madrid, 28040 Spain.}
\email{pablo.berna@uam.es}



\begin{document}

\begin{abstract} 
Our aim in this paper is to attain a sharp asymptotic estimate for the greedy constant   $C_g[\HH^{(p)},L_p]$ of the (normalized) Haar system $\HH^{(p)}$ in  $L_{p}[0,1]$ for $1<p<\infty$.  We will show that the superdemocracy constant of  $\HH^{(p)}$ in $L_{p}[0,1]$ grows as  $p^{\ast}=\max\{p,p/(p-1)\}$ as $p^*$ goes to $\infty$.  Thus, since the unconditionality constant of $\HH^{(p)}$ in $L_{p}[0,1]$  is $p^*-1$,
 the  well-known  general estimates for the greedy constant of a greedy basis obtained  from the intrinsic features of greediness (namely, democracy and unconditionality) yield that  $p^{\ast}\lesssim C_g[\HH^{(p)},L_p]\lesssim (p^{\ast})^{2}$. Going further, we develop techniques that allow
 us to close the gap between those two bounds,  establishing that, in fact,  $C_g[\HH^{(p)},L_p]\approx p^{\ast}$.

\end{abstract}

\maketitle

\section{Introduction}\label{Introduction}
\noindent
A fundamental and total biorthogonal system for an infinite-dimen\-sional separable Banach space $(\XX, \Vert\cdot\Vert)$ over the field $\FF$ of real or complex scalars, is a family $(\xx_{j},\xx_{j}^{\ast})_{j\in J}$ in $\XX\times \XX^{\ast}$ verifying
\begin{itemize}
\item[(i)] $\XX= \overline{\spn \{ \xx_{j} : j\in J\}}$,
\item[(ii)] $\XX^*= \overline{\spn \{ \xx^*_{j} : j\in J\}}^{w^*}$,
 and
\item[(iii)] $\xx_{j}^{\ast}(\xx_{k})=1$ if $j=k$ and $\xx_{j}^{\ast}(\xx_{k})=0$ otherwise.
\end{itemize}
The family $\BB=(\xx_j)_{j\in J}$ is called a \textit{(Markushevich) basis} and the unequivocally determined collection of bounded linear functionals $\BB^*=(\xx_j^*)_{j\in J}$ is said to be the family of \textit{coordinate functionals} (or \textit{dual basis}) of $\BB$. If the biorthogonal system verifies the condition
\begin{itemize}
\item[(iv)] $\sup_{j\in J} \Vert \xx_j\Vert \Vert \xx^*_j\Vert <\infty$
\end{itemize}
we say that $\BB$ is $M$-\textit{bounded}. Finally, if we have
\begin{itemize}
\item[(v)] $0<\inf_{j\in J} \Vert \xx_j\Vert \le \sup_{j\in J} \Vert \xx_j\Vert <\infty$
\end{itemize}
we say that $\BB$ is \textit{semi-normalized} (\textit{normalized} if $\Vert \xx_j\Vert=1$ for all $j\in J$). Note that a basis $\BB$ verifies simultaneously (iv) and (v) if and only if
\[
\textstyle
\sup_{j \in J} \max\{ \Vert \xx_j\Vert ,\Vert \xx_j^*\Vert \}<\infty.
\]

Suppose $\BB=(\xx_j)_{j\in J}$ is a semi-normalized $M$-bounded basis in a Banach space $\XX$ with coordinate functionals $(\xx_j^*)_{j\in J}$. Then each $f\in \XX$ is uniquely determined by its \textit{coefficient family} $(\xx_j^*(f))_{j\in J}$, which belongs to $c_0(J)$.
Thus, we can consider its non-increasing rearrangement,  which we denote by $(a_m^*[\BB,\XX](f))_{m=1}^\infty$ (or, simply $(a_m^*(f))_{m=1}^\infty$ if the basis and the space are clear from context). 

For each $f\in \XX$ there is a $1-1$ map $\rho\colon \NN\to J$ such that 
\begin{equation}
\label{greedycondition}
|\xx_{\rho(m)}^{\ast}(f)| = a_m^*(f), \quad m\in\NN.
\end{equation}
If the family $(\xx_j^*(f))_{j\in J}$ contains several terms with the same absolute value then the map $\rho$ for $f$ is not uniquely determined. In order to get uniqueness, we arrange the countable set $J$ by means of a bijection  $\sigma\colon J\to\NN$ and impose the additional condition
\begin{equation}
\label{desempate}
\sigma(\rho(m))\le \sigma(\rho(n)) \text{ whenever }\quad |\xx_{\rho(m)}^{\ast}(f)|= |\xx_{\rho(n)}^{\ast}(f)|.
\end{equation}
If $f$ is infinitely supported there is a unique $1-1$ map $\rho\colon\NN\to J$ with $\rho(\NN)=\supp (f)$ that verifies \eqref{greedycondition} and \eqref{desempate}.
In the case when $f$ is finitely supported, there is a unique bijection $\rho\colon\NN\to J$ that verifies \eqref{greedycondition} and \eqref{desempate}.
 In any case, we will refer to such a map $\rho$ as the \textit{greedy ordering} for $f$. For each $m\in \NN$, the \textit{$m$th-greedy approximation}  to $f$ is the partial sum 
 
\[
\GG_{m}[\BB, \XX](f):=\GG_{m}(f) = \sum_{n=1}^{m}\xx_{\rho(n)}^{\ast}(f) \, \xx_{\rho(n)},
\]
where $\rho$ is the greedy ordering for $f$. The sequence $(\GG_m(f))_{m=1}^\infty$ is called the greedy algorithm for $f$ with respect to the basis $\BB$. 

To quantify the efficiency of the greedy algorithm, Temlyakov \cites{Temlyakov2008} introduced 
the sequence $(\Leb_m)_{m=1}^\infty$ of \textit{Lebesgue greedy constants}. For each $m\in\NN$,  $ \Leb_m[\BB,\XX]:=\Leb_m$ 
is the smallest constant  $C$ such that the Lebesgue-type inequality
\begin{equation}\label{LebGreedIneq}
\Vert f-\GG_m(f)\Vert \le C \left\Vert f -\sum_{j\in A} a_j \, \xx_j\right\Vert, \end{equation}
holds for all  $f\in \XX$, all  subsets $A$ of $\mathbb N$ with $|A|=m$,  and all $a_j\in\FF$.
Konyagin and Temlyakov \cite{KonyaginTemlyakov1999} then defined a  basis $\BB$ to be \textit{greedy} if \eqref{LebGreedIneq} holds with a constant $1\le C< \infty$ independent of $m$. The smallest admissible constant $C$ will be  denoted by $C_g[\BB,\XX]=C_g$, and will be referred to as the  \textit{greedy constant}  of the basis. In other words, a basis $\BB$ is greedy if and only if
  $\sup_m \Leb_m=C_g<\infty$, that is, the greedy algorithm provides, up to a multiplicative constant, the best $m$-term approximation to any vector in the space.

Once we know that a certain basis is greedy, a natural problem in approximation theory is to compute, or at least estimate, its greedy constant. Also of interest is to determine for what values of $C$ a basis is $C$-greedy under a suitable renorming of the space.

In this paper we focus on the Haar system  $\HH^{(p)}=(h_I^{(p)})_{I\in\DDD_0}$ in the spaces $L_p[0,1]$ ($L_p$ for short from now on).   Here, $\DDD_0$ denotes the set $\DDD\cup\{0\}$, where $\DDD$ is the collection of all  dyadic intervals contained in $[0,1)$, $h_0^{(p)}$ is the constant function  $1$ on $[0,1)$, and $h_I^{(p)}$ stands for the $L_p$-normalized Haar function supported on $I$, i.e., if $I=[a,b)$ then
\[
h_I^{(p)}(t)=\begin{cases} 
- 2^{-(b-a)/p} &\text{ if } a\le t <(a+b)/2, \\ 
2^{-(b-a)/p} &\text{ if } (a+b)/2\le t <b.
\end{cases}
\]
 Thus $\HH^{(p)}$ is a normalized $M$-bounded basis  for $L_p$  when $1\le p<\infty$ and a normalized $M$-bounded basis  for the space $\DD$ of C\`agl\`ag functions when $p=\infty$.  In either case, the family of coordinate functionals of   $\HH^{(p)}$ is $\HH^{(p')}$
 (with the canonical isometric identification of functions in $L_{p^\prime}$ with
functionals in $(L_p)^*$, where  $p'={p}/{(p-1)}$). In fact, when arranged in the natural way, $\HH^{(p)} $ is a Schauder basis (see \cite{AlbiacKalton2016}*{Proposition 6.1.3}).
  
Temlyakov \cite{Temlyakov1998} showed that $\HH^{(p)}$ is a greedy basis in $L_p$ for $1<p<\infty$.  Later on, Dilworth et al. \cite{DKOSZ2014} proved that for every $C>1$ there is a renorming of $L_p$ with respect to which $\HH^{(p)}$ is $C$-greedy. Whether  or not the isometric constant $C=1$ can be achieved up to renorming (see  \cite{AW2006})  remains  unknown as of today. 

Note that  neither $\HH^{(1)}$ is  a greedy basis for $L_1$ nor $\HH^{(\infty)}$ is  a greedy basis for $\DD$. Indeed, every greedy (or unconditional) basis for a $\LL_1$-space is equivalent to the unit vector basis of $\ell_1$ and every greedy basis for a $\LL_\infty$-space is equivalent to the unit vector basis of $c_0$.
Consequently we have
 \[
 \lim_{p\to 1^+} C_g[\HH^{(p)},L_p]=\infty= \lim_{p\to\infty} C_g[\HH^{(p)},L_p].
 \]
 
 The Haar system  $\HH^{(2)}$ is an orthonormal basis for $L_2$, which easily yields $C_g[\HH^{(2)},L_2]=1$.  However,  for $p\not=2$  it seems hopeless to  attempt to compute the exact value of $C_g[\HH^{(p)},L_p]$. It is therefore natural to address the problem of obtaining asymptotic estimates for $C_g[\HH^{(p)},L_p]$ as $p$ tends to $1$ or to $\infty$. 
 
A standard approach  to estimate the greedy constant of a greedy basis is to make use of its intrinsic properties instead of relying on the mere definition.  The first movers in this direction were Konyagin and Temlyakov \cite{KonyaginTemlyakov1999}, who proved that a basis is greedy if and only if it is \textit{unconditional} and \textit{democratic}.  To set the notation, we recall that a basis $\BB=(\xx_j)_{j\in J}$ for a Banach space $\XX$ is said to be unconditional if the series expansion $\sum_{j\in J} \xx_j^*(f) \, \xx_j$ converges unconditionally to $f$ for every $f\in\XX$. Unconditional bases are characterized as those bases verifying the uniform bound
\begin{equation}\label{SupUncDef}
K_{su}[\BB,\XX]:=\sup_{\substack{A\subset J \\ |A|<\infty}} 
\Vert S_A\Vert
=\sup_{\substack{A\subset J \\  |A|<\infty}} \Vert  \Id_\XX-S_A\Vert
 <\infty,
\end{equation}
or, equivalently,
\begin{equation}\label{LatUncDef}
K_{u}[\BB,\XX]:=\sup
\{ \Vert M_\varepsilon\Vert \colon \varepsilon=(\varepsilon_j)_{j\in J},\, |\varepsilon_j|=1\} <\infty,
\end{equation}
 where $S_A=S_A[\BB,\XX]$ is the coordinate projection on a finite set $A\subseteq J$, i.e.,
\[
\textstyle
S_A\colon \XX \to \XX, \quad f\mapsto\sum_{j\in A} \xx_j^*(f)\, \xx_j,
\]
and
$M_\varepsilon = M_\varepsilon[\BB,\XX]$ is the linear operator from $\XX$ into $\XX$ given by $\xx_j \mapsto \varepsilon_j \,\xx_j$. The {\it suppression unconditional constant} $K_{su}$ and the \textit{lattice unconditional constant} $K_{u}$ of a basis  are related by the inequalities
\[
K_{su}[\BB,\XX]\le K_{u}[\BB,\XX] \le \kappa\, K_{su}[\BB,\XX],
\]
where $\kappa=2$ if  $\FF=\RR$ and $\kappa= 4$  if $\FF=\CC$.
In turn, a basis $\BB$ is said to be \textit{democratic}  if there is $1\le C<\infty$ such that
\begin{equation}\label{DefDemocracy}
\left\Vert \sum_{j\in A} \xx_j \right\Vert \le C \left\Vert \sum_{j\in B} \xx_j \right\Vert, \quad |A|=|B|<\infty.
\end{equation}
We will denote by $\Delta[\BB,\XX]$ the optimal constant $C$ in \eqref{DefDemocracy}.
By imposing  the extra assumption $A\cap B=\emptyset$ in \eqref{DefDemocracy}  we obtain an equivalent definition of democracy,
and $\Delta_{d}[\BB,\XX]$ will denote the optimal constant  in \eqref{DefDemocracy} under the extra assumption on disjointness of sets.

 Amalgamating some steps in Konyagin--Temlyakov's proof (see \cite{DKOSZ2014}*{Equation 1}) 
we get the estimate 
\begin{equation}\label{UncDemEstimate}
\max\{ K_{su}[\BB,\XX], \Delta[\BB,\XX] \} \le C_g[\BB,\XX]\le K_{su}[\BB,\XX] + \Delta_{d}[\BB,\XX] K_u^2[\BB,\XX].
\end{equation}

Note that 
\[\Delta[\BB,\XX]\le \Delta_{d}[\BB,\XX](1+K_{su}[\BB,\XX]).\]
 Hence, when $\BB$ runs over a certain family of bases,
the left-hand side  and the right-hand side terms in inequality \eqref{UncDemEstimate} are of the same order if and  only if the constants $K_u[\BB,\XX]$ are uniformly bounded. This is not the case for   $(\HH^{(p)})_{p>1}$ as the following   theorem of Burkholder exhibits:

\begin{theorem}[\cite{Burk}]\label{UnconditionalEstimate} If $1<p<\infty$ then $K_{u}[\HH^{(p)},L_p]=p^*-1$, where $p^*=\max\{p,p'\}$.
\end{theorem}

Now, we may try to obtain estimates that bring us closer to our goal by using super-democracy instead of democracy. A basis is \textit{super-democratic} if there is a constant $1\le C<\infty$ such that 
\begin{equation}\label{SuperDemDefi}
\left\Vert \sum_{j\in A} \varepsilon_j \, \xx_j \right\Vert \le C \left\Vert \sum_{j\in B} \delta_j \, \xx_j \right\Vert, \quad |A|=|B|<\infty, \, |\varepsilon_j|=|\delta_j|=1.
\end{equation}
The smallest admissible constant in \eqref{SuperDemDefi} will be denoted by $\Delta_{s}[\BB,\XX]$, and 
the smallest constant  in \eqref{SuperDemDefi} with the extra assumption $A\cap B=\emptyset$ will be denoted by $\Delta_{sd}[\BB,\XX]$.

Bases that are  unconditional and democratic are super-democratic. Hence, greedy bases are characterized as those that are simultaneously unconditional and super-democratic. Quantitatively, a slight improvement of the argument  used in the proof of \cite{BBG2017}*{Theorem 1.4} gives \begin{equation*}
\max\{ K_{su}[\BB,\XX], \Delta_{sd}[\BB,\XX] \} \le C_g[\BB,\XX]\le K_{su}[\BB,\XX](1+ \Delta_{sd}[\BB,\XX]).
\end{equation*}
 These inequalities allow us to determine the rate of growth of the constants $C_g[\BB,\XX]$ 
when $\BB$ runs over  a family of bases only when $\min\{[K_{su}[\BB,\XX],\Delta_{sd}[\BB,\XX]\}$ is uniformly bounded. But, again, this is not the case for the $L_p$-normalized Haar system for $1<p<\infty$. Indeed, on the one hand we have $\Delta_{sd}[\BB,\XX] \ge \Delta_{d}[\BB,\XX]$ for any $\BB$ and any  $\XX$.  On the other hand,  the following result (which  we shall prove below) yields $\sup_{p>1} \Delta_{d}[\HH^{(p)},L_p]=\infty$.
\begin{proposition}\label{LowerEstimateHaar} If $1<p<\infty$ then
\[
\Delta_{d}[\HH^{(p)},L_p]\ge d_p:= \frac{2^{1/p^\#}-1}{2^{1/p^*}-1},
 \]
 where $p^\#=\min\{p,p'\}.$

\end{proposition}
\noindent

Another important property of bases that comes into play in this scenario is the \textit{symmetry for largest coefficients} (a.k.a.\ \textit{Property A}). A basis $\BB$ is said to be symmetric for largest coefficients  
 is there is a constant $1\le C<\infty$ such that
\begin{equation}\label{SSLDefi}
\left\Vert \sum_{j\in A} \varepsilon_j \, \xx_j +f \right\Vert \le C \left\Vert \sum_{j\in B} \delta_j \, \xx_j + f\right\Vert
\end{equation}
whenever
$|A|=|B|<\infty$, $A\cap B=(A\cup B)\cap\supp(f)=\emptyset$, and $ |\xx_i^*(f)| \le |\varepsilon_j|=|\delta_k|=1$ for all $i\in J$, $j\in A$, $k\in B$. 
We denote by $C_a[\BB,\XX]$ the optimal constant $C$ in \eqref{SSLDefi}.
A basis is greedy if and only if it is unconditional and symmetric for largest coefficients. Moreover (see \cite{AA2017bis}*{Remark 2.6}),
 \begin{equation}\label{UncSLCEstimate}
\max\{ K_{su}[\BB,\XX], C_a[\BB,\XX] \} \le C_g[\BB,\XX]\le C_a[\BB,\XX] K_{su}[\BB,\XX].
\end{equation}
These estimates are useful when one wants to show  that the greedy constant of a certain basis is close to $1$. However, since
$
C_{a}[\BB,\XX]\ge \Delta_{sd}[\BB,\XX], 
$
and the side terms of \eqref{UncSLCEstimate} are of different order, they do not provide a tight information about the asymptotic growth of the greedy constants of a  family of bases.

Despite the fact that the  methods  described above are not strong enough to be applied to our problem, in this note we  shall  reach our goal and prove the following theorem.

 \begin{theorem}\label{ABGAHS} 
 $C_g[\HH^{(p)},L_p]\approx p^*$ for $1<p<\infty$.
 \end{theorem}

 The key idea in the proof of Theorem~\ref{ABGAHS} will consist of taking advantage of the fact that the Haar system in $L_p$ belongs to a more demanding class of bases than that of greedy bases, namely, the class of \textit{bi-greedy} bases. A basis is said to be bi-greedy if both the basis itself and its dual basis are greedy. Bi-greedy bases were characterized in \cite{DKKT2003} as those bases that are unconditional and \textit{bi-democratic}. Recall that a basis $\BB=(\xx_j)_{j=1}^\infty$ is said to be bi-democratic if there is a constant $1\le C<\infty$ such that 
  \begin{equation}\label{BiDemDefi}
 \left\Vert \sum_{j\in A} \xx_j\right\Vert \, \left\Vert \sum_{k\in B} \xx_k^*\right\Vert \le C m, \quad |A|=|B|=m.
 \end{equation}
We will denote by $\Delta_{b}[\BB,\XX]$ 
the smallest constant $C$ such that \eqref{BiDemDefi} holds, and we will refer to it
 as the bi-democratic constant of the basis. 

 The following new estimate for the greedy constant will also be crucial in the proof of Theorem~\ref{ABGAHS}.
\begin{theorem}\label{BidemocracyEstimate}Let $\BB$ be a bi-democratic and unconditional basis for a Banach space $\XX$.
Then
\[
C_g[\BB,\XX]\le K_{su}[\BB,\XX] + \kappa^2 \, \Delta_{b}[\BB,\XX].
\]
\end{theorem}
Section~\ref{BidemocraticSection} is devoted to proving Theorem~\ref{BidemocracyEstimate}, while in Section~\ref{HaarEstimatesSection} we obtain the remaining estimates that in combination with Theorem~\ref{BidemocracyEstimate} will yield Theorem~\ref{ABGAHS}.

Throughout this article we use standard facts and notation from Banach spaces and approximation theory. Here, and throughout this paper, the symbol $\alpha_i\lesssim \beta_i$ for $i\in I$ means 
that the families of positive real numbers $(\alpha_i)_{i\in I}$ and $(\beta_i)_{i\in I}$ verify $\sup_{i\in I}\alpha_i/\beta_i <\infty$. If $\alpha_i\lesssim \beta_i$ and $\beta_i\lesssim \alpha_i$ for $i\in I$ we say $(\alpha_i)_{i\in I}$ are $(\beta_i)_{i\in I}$ are equivalent, and we write $\alpha_i\approx \beta_i$ for $i\in I$.
We refer the reader to  \cites{AlbiacKalton2016} for the necessary background. 

\section{A new estimate for the Lebesgue type constants for the greedy algorithm using bi-democracy}
\label{BidemocraticSection}
\noindent
The \textit{fundamental function} of a basis $\BB=(\xx_j)_{j=1}^\infty$ in a Banach space $\XX$ is the sequence given by
\[
\varphi_m[\BB,\XX]=\sup_{|A|\le m}\left\Vert \sum_{j\in A} \xx_j\right\Vert, \quad m\in\NN.
\]
We shall also consider the \textit{super-fundamental function} of the basis, given by
\[
\varphi^{\epsilon}_m[\BB,\XX]=\varphi^{\epsilon}_m:= \sup_{\substack{|A|=m\\ |\varepsilon_j|=1}}\left\Vert \sum_{j\in A} \varepsilon_j \xx_j\right\Vert, \quad m\in\NN.
\]
A standard convexity argument yields
\begin{equation}\label{eq:FundamentalFunction}
\varphi_m[\BB,\XX]
\le \sup_{\substack{|A|=m\\ |a_j|\le 1}}\left\Vert \sum_{j\in A} a_j \xx_j\right\Vert
=\varphi^{\epsilon}_m[\BB,\XX]
\le \kappa \, \varphi_m[\BB,\XX].
\end{equation}

If $\YY$ is the subspace of $\XX^*$ spanned by $\BB^*$, we set
\[
\varphi^{\epsilon,*}_m[\BB,\XX]=\varphi^{\epsilon,*}_m:=\varphi^{\epsilon}_m[\BB^*,\YY], \qquad m\in \NN,\] and define the sequence
\[
\Bid_m[\BB,\XX]=\Bid_m
:=\sup_{r\le m} \frac{\varphi^{\epsilon}_r[\BB,\XX] \, \varphi^{\epsilon,*}_r[\BB,\XX] }{r}.
 \]
Then a basis $\BB$ is bi-democratic if and only 
\begin{equation}\label{CSB}
\Delta_{sb}[\BB,\XX]:=\sup_m \Bid_m<\infty.
\end{equation} Quantitatively we have
 \begin{equation} \label{BidemocracyConstantsRelation}
\Delta_{b}[\BB,\XX]\le \Delta_{sb}[\BB,\XX]\le \kappa^2 \, \Delta_{b}[\BB,\XX].
\end{equation}
\begin{remark}\label{FromBiDemToDem}The identity
 \[
 \left(\sum_{j\in A}\xx_j^*\right)\left(\sum_{j\in A}\xx_j\right)=|A|, \quad A\subseteq J, \, |A|<\infty
 \] 
 yields (see \cite{DKKT2003})
 \[
 \Delta[\BB,\XX]\le \Delta_{b}[\BB,\XX]\]  Similarly, the identity  
 \[
 \left(\sum_{j\in A}\varepsilon_n \, \xx_j^*\right)\left(\sum_{j\in A} \overline{\varepsilon_n}\, \xx_j\right)=|A|,  \quad A\subseteq J, \, |A|<\infty, \, |\varepsilon_j|=1
 \]
gives 
\[\Delta_{s}[\BB,\XX]\le \Delta_{sb}[\BB,\XX].\]
\end{remark}

Garrig\'os et al.\ \cite{GHO2013} gave several estimates for the constants $\Leb_m[\BB,\XX]$ involving the sequences of democracy constants and of conditionality constants of the basis. Subsequently, Bern\'a et al. \cite{BBG2017} obtained estimates for $\Leb_m[\BB,\XX]$ that also involved the sequence of quasi-greedy constants of the basis. The estimate we shall provide involves the sequence of bi-democracy constants and the sequence of conditionality constants $(\kk_m^c)_{m=1}^\infty$ given by
\[
\textstyle
\kk_m^c[\BB,\XX]=\kk_m^c:=\sup_{|A|\le m}\Vert \Id_\XX-S_A[\BB,\XX]\Vert.
\]
Note that the basis $\BB$ is unconditional if and only if 
\[\sup_m \kk_m^c=K_{su}[\BB,\XX]<\infty.\]
\begin{lemma}\label{p1}
Let $\BB=(\xx_j)_{j\in J}$ be a semi-normalized M-bounded basis in a Banach space $\XX$.
For every $f\in \XX$ and $m\in\NN$,
\[
a_m^*[\BB,\XX] (f) \, \varphi^{\epsilon}_m[\BB,\XX] \leq \Bid_m[\BB,\XX] \, \Vert f\Vert.
\]

\end{lemma}
\begin{proof}
Let $G\subseteq J$ be such that $|G|=m$ and $a_m^*\le |\xx_j^*(f)|$ for all $j\in G$.
Define $f^*\in \XX^*$ by 
$
f^*=\sum_{j\in G} \overline{\sgn \xx_j^*(f)} \, \xx_j^*.
$
Then
\[
a_m^*(f) \varphi^{\epsilon}_m
 \leq \Bid_m \dfrac{m \, a_m^*(f)}{\varphi_m^{\epsilon,*}}
\leq 
 \Bid_m \dfrac{\sum_{j\in G} |\xx_j^*(f)|}{\Vert f^* \Vert}
 =
 \Bid_m \dfrac{f^*(f)}{\Vert f^* \Vert}
 \leq \Bid_m\Vert f\Vert.
\]

\end{proof}

\begin{theorem}\label{p2}
Let $\BB=(\xx_j)_{j\in J}$ be a semi-normalized M-bounded basis in a Banach space $\XX$. For all $m\in\NN$ we have
\[
\Leb_m[\BB,\XX] \leq \kk^c_{2m}[\BB,\XX] + \Bid_m[\BB,\XX].
\]
\end{theorem}
\begin{proof}
Let $f\in \XX$, $m\in\NN$ and $G\subseteq J$ of cardinality $m$ such that
$\GG_m(f)=S_G(f)$. Let $A\subseteq J$ with $|A|=m$ and $(a_j)_{j\in A}\in\FF^A$. Put $g=f-\sum_{j\in A} a_j\, \xx_j$.
We have
\[
\Vert f-\GG_m(f)\Vert 
=\Vert g-S_{A\cup G}(g)+S_{A\setminus G}(f)\Vert
\le \Vert g-S_{A\cup G}(g)\Vert +\Vert S_{A\setminus G}(f)\Vert.
\]
Since $|A\cup B|\le 2 m$,
\[\Vert g-S_{A\cup G}(g)\Vert\le \kk_{2m}^c\, \Vert g\Vert.\]
Let $r=|A\setminus G|=|G\setminus A|$. Invoking \eqref{eq:FundamentalFunction} and Lemma~\ref{p1},
\begin{align*}
\Vert S_{A\setminus G}(f)\Vert 
&\leq
\max_{j\in A\setminus G}\vert \xx_j^*(f)\vert \varphi^{\varepsilon}_r\\
&\leq
\min_{j\in G\setminus A}\vert \xx_j^*(f)\vert \varphi^{\varepsilon}_r\\
&=\min_{j\in G\setminus A}\vert \xx_j^*(g)\vert \varphi^{\varepsilon}_r\\
&\le a_r^*(g) \varphi^{\varepsilon}_r\\
&\le \Bid_r \, \Vert g\Vert.
\end{align*}
 Taking into account that $r\le m$ and that $(\Bid_m)_{m=1}^\infty$ is non-decreasing, we get 
 \[
 \Vert S_{A\setminus G}(f)\Vert \le \Bid_m \, \Vert g\Vert.\]
Combining, we obtain the desired result.
\end{proof}

\begin{proof}[Proof of Theorem~\ref{BidemocracyEstimate}] Taking the supremum over $m$ in Theorem~\ref{p2}  and appealing to
\eqref{BidemocracyConstantsRelation} gives
$$
C_g[\BB,\XX]\le K_{su}[\BB,\XX]+\Delta_{sb}[\BB,\XX]\le K_{su}[\BB,\XX]+\kappa^2\, \Delta_{b}[\BB,\XX].
$$
 
\end{proof}

\section{Estimates for the Haar basis in $L_p$}\label{HaarEstimatesSection}
\noindent
We start this section with the proof   advertised in Section~\ref{Introduction} of the lower estimate for the democracy constant.
\begin{proof}[Proof of Proposition~\ref{LowerEstimateHaar}]
If $(J_j)_{j=1}^m$ are disjointly supported intervals in $\DDD$ we have
\begin{equation}\label{Estimate1}
\left\Vert \sum_{j=1}^m h_{J_j}^{(p)}\right\Vert_p=m^{1/p}.
\end{equation}
Let $(I_j)_{j=1}^\infty$ be the sequence in $\DDD$ defined recursively as follows:
$I_1=[0,1)$ and
$I_{j+1}$ is the left half of $I_j$.
Set $q=p'$.Then
\begin{align*}
\left\Vert \sum_{j=1}^m h_{I_j}^{(p)}\right\Vert_p^p
&=2^{-m-1}\left|\sum_{k=0}^{m} 2^{k/p}\right|^p + \sum_{j=0}^{m-1}2^{-j-1}\left| 2^{j/p}-\sum_{k=0}^{j-1} 2^{k/p}\right|^p\\
&=\frac{(1-2^{-(m+1)/p})^p+\sum_{j=0}^{m-1} |- 2^{-(j+1)/p}+(2^{1/q}-1)|^p}{(2^{1/p}-1)^p}\\
&=\frac{\Vert f-g\Vert_p^p}{(2^{1/p}-1)^p},
\end{align*}
where
\[
 g=( 2^{-(j+1)/p})_{j=0}^m, \quad f=(\underbrace{2^{1/q}-1, \dots , 2^{1/q}-1}_{m \text{ times}} ,1).
\]
We have $\Vert g\Vert_p\le 1$ and $\Vert f\Vert_p=(1+m (2^{1/q}-1)^p)^{1/p}$.
 Hence, by Minkowski's inequality,
\begin{equation}\label{Estimate2}
 \frac{(1+m (2^{1/q}-1)^p)^{1/p} - 1}{2^{1/p}-1}
\le \left\Vert \sum_{j=1}^m h_{I_j}^{(p)}\right\Vert_p
\le \frac{(1+m (2^{1/q}-1)^p)^{1/p} + 1}{2^{1/p}-1}.
\end{equation}
Comparing \eqref{Estimate1} with \eqref{Estimate2}, and letting $m$ tend to $\infty$, we get
\[
\Delta_{d}[\HH^{(p)},L_p] \ge \max\left\{ \frac{2^{1/q}-1}{2^{1/p}-1}, \frac{2^{1/p}-1}{2^{1/q}-1}\right\},
\]
as desired.
\end{proof}

Next we establish the upper estimate  for the super-bi-democracy constants  that we will need.
\begin{proposition}\label{UpperEstimateHaar} If $1<p<\infty$ then
\begin{equation}\label{Dp}
\Delta_{sb}[\HH^{(p)},L_p]\le D_p:= \frac{8}{(2^{1/p}-1)(2^{1/p'}-1)}.
\end{equation}
\end{proposition}
\begin{proof}
For $I\in\DDD$ let $n(I)$ be such that $|I|=2^{-n(I)}$.
Let $A\subseteq \DDD$ finite and $\varepsilon=(\varepsilon_I)_{I\in A}$ be such that $|\varepsilon_I|=1$ for all $I\in A$.
For $J\in A$ set
\[
 R_ J=J\setminus \cup\{ I \colon I\in A, n(I)>n(J) \}.
\] 
Taking into account that, for $n\in\NN$, the collection of dyadic intervals $\{I\in\DDD \colon n(I)=n\}$ is a partition on $[0,1)$ we infer that
\begin{itemize}
\item $(R_I)_{I\in A}$ is a partition of $K=\cup_{I\in A} I$, 
\item $R_I\subseteq I$ for every $I\in A$, and
\item given $t\in K$ and $k\in\NN$ there is at most one
 interval $I_{t,k}\in A\cap\DDD_k$ such that $t\in I_{t,k}$; moreover $n(I_{t,k}) \le n(J)$.
\end{itemize}
Consequently, for any $t\in K$ we have
\begin{align*}
\left|\sum_{I\in A} \varepsilon_I h_I^{(p)}(t)\right|
&\le \sum_{I\in A} \left\vert h_I^{(p)}(t)\right\vert\\
&=\sum_{I\in A} 2^{n(I)/p} \chi_I(t)\\
&\le \sum_{J\in A}\left( \sum_{n=-\infty}^{n(J)} 2^{n/p}\right) \chi_{R_J}(t)\\
&= 
\frac{1}{1-2^{-1/p}} \sum_{J\in A} 2^{n(J)/p}\chi_{R_J}(t).
\end{align*}
Hence,  if we set $a_p=1/(1-2^{-1/p})$ we obtain
\begin{align*}
\left\Vert \sum_{I\in A} \varepsilon_I h_I^{(p)}\right\Vert_p
&\le a_p \left(\sum_{J\in A} 2^{n(J)} |R_J|\right)^{1/p}\\
&\le a_p \left(\sum_{J\in A} 2^{n(J)} |J|\right)^{1/p}\\
&= a_p |A|^{1/p}.
\end{align*}
Therefore, 
\[
\left\Vert h_0^{(p)}+ \sum_{I\in A} \varepsilon_I h_I^{(p)}\right\Vert_p\le  1+a_p |A|^{1/p}.\] We infer that, for  $m\in\NN$,
\[
\varphi^{\epsilon}_m[\HH^{(p)},L_p]\le \max\{ a_p m^{1/p}, 1+a_p (m-1)^{1/p}\}\le 2 a_p m^{1/p}.
\]
The fact that $\varphi^{\epsilon,*}_m[\HH^{(p)},L_p]=\varphi^{\epsilon}_m[\HH^{(p')},L_{p'}]$ for 
 $m\in\NN$ yields
\[
\frac{\varphi^{\epsilon}_m[\HH^{(p)},L_p]\, \varphi^{\epsilon,*}_m[\HH^{(p)},L_p]}{m}\leq 4 a_p a_{p'}
\frac{m^{1/p}m^{1/p'}}{m}=4  a_p a_{p'}=D_p.
\]
Thus $\Delta_{sb}[\HH^{(p)},L_p]\le D_p$.
\end{proof}
\begin{corollary}\label{AsymptoticDemocracy} We have
\begin{align*}
\Delta_{b}[\HH^{(p)},L_p] & \approx \Delta_{s}[\HH^{(p)},L_p]
 \approx \Delta_{sd}[\HH^{(p)},L_p]
\approx \Delta[\HH^{(p)},L_p] 
\approx \Delta_{d}[\HH^{(p)},L_p]\\
&\approx p^*
\end{align*}
for $1<p<\infty$.

\end{corollary}
\begin{proof} Let $D_p$ and $d_p$ be  as in Proposition~\ref{LowerEstimateHaar} and Propostion~\ref{UpperEstimateHaar}. We have $d_p\approx D_p\approx p^*$ for $1<p<\infty$. Then the result  follows by combining Proposition~\ref{LowerEstimateHaar}, Propostion~\ref{UpperEstimateHaar}, and Remark~\ref{FromBiDemToDem}.
\end{proof}

We close by providing the conclusion of the proof of our main theorem and
enunciating a corollary. 
\begin{proof}[Conclusion of the Proof of Theorem~\ref{ABGAHS}] Combine Theorem~\ref{BidemocracyEstimate} with the left-hand side of inequality  \eqref{UncDemEstimate},  Theorem~\ref{UpperEstimateHaar}
and Theorem~\ref{UnconditionalEstimate}.
\end{proof}

 \begin{corollary}  $C_a[\HH^{(p)},L_p]\approx p^*$ for $1<p<\infty$.
 \end{corollary}
 
 \begin{proof}Just combine the left-hand side inequality in \eqref{UncSLCEstimate}, with Corollary~\ref{AsymptoticDemocracy} and Theorem~\ref{ABGAHS}.
 \end{proof}
 
 \section*{Annex: Summary of  the most commonly employed constants}
 
 \begin{table}[ht]
\begin{center}
\begin{tabular}{c c c}\hline
& & \\
{\bf Symbol}& {\bf Name of constant} &  {\bf Ref. equation}\\
& & \\
 \hline
 & & \\
$C_{a}$ & Symmetry for largest coeffs. constant & \eqref{SSLDefi} \\ 
& & \\
$C_{g}$ & Greedy constant & \eqref{LebGreedIneq} \\ 
             & & \\
$\Delta $ & Democracy constant & \eqref{DefDemocracy} \\ 
             & & \\
$\Delta_{b}$         & Bi-democracy constant & \eqref{BiDemDefi} \\
  & & \\
 $\Delta_{d}$ & Disjoint-democracy constant & \eqref{DefDemocracy}) \\ 
 & & \\
$\Delta_{s}$ & Superdemocracy constant &  \eqref{SuperDemDefi} \\ 
& & \\
$\Delta_{sb}$ & Super bi-democratic constant & \eqref{CSB} \\ 
& & \\
$\Delta_{sd}$ & Disjoint-superdemocracy constant & \eqref{SuperDemDefi} \\             
             & & \\
$K_{su}$ & Suppression unconditional constant & \eqref{SupUncDef} \\ 
             & & \\
$K_{u}$ & Lattice unconditional constant & \eqref{LatUncDef} \\ 
            \end{tabular}
 \end{center}
\end{table}

\subsection*{Acknowledgments} F. Albiac and J. L. Ansorena acknowledge the support of the grant MTM2014-53009-P (MINECO, Spain). F. Albiac was also  supported by the grant    MTM2016-76808-P (MINECO, Spain). P. Bern\'a was supported by a Ph.D.\ fellowship from the program FPI-UAM, as well as the grants MTM-2016-76566-P (MINECO, Spain) and 19368/PI/14 (\emph{Fundaci\'on S\'eneca}, Regi\'on de Murcia, Spain).


\end{document}